\documentclass[a4paper,12pt]{amsart}
\usepackage{amsmath}
\usepackage{amssymb}
\usepackage{amscd}
\usepackage{mathrsfs}
\usepackage[all]{xy}
\usepackage{graphicx}
\usepackage[usenames]{color}
\usepackage{hyperref}
\hypersetup{
 colorlinks,
 citecolor=blue,
 filecolor=blue,
 linkcolor=blue,
 urlcolor=blue
}

\def\mathcal{\mathscr}
\everymath{\displaystyle}
\theoremstyle{plain}
\newtheorem{theorem}{\indent\sc Theorem}[section]
\newtheorem{lemma}[theorem]{\indent\sc Lemma}

\newtheorem{proposition}[theorem]{\indent\sc Proposition}

\theoremstyle{definition}

\newtheorem{remark}[theorem]{\indent\sc Remark}

\newcommand{\om}{\omega}
\newcommand{\p}{\partial}
\newcommand{\R}{\mathbb R}
\newcommand{\Core}{\mathrm{Core}}
\newcommand{\Int}{\mathrm{Int}\,}
\newcommand{\Op}{\mathcal{O}p\,}
\newcommand{\wt}{\widetilde}
\newcommand{\wh}{\widehat}

\newcommand{\Span}{\mathrm{Span}\,}
\newcommand{\ol}{\overline}

\begin{document}
\pagestyle{plain}
\thispagestyle{plain}

\title[Stabilized convex symplectic manifolds are Weinstein]{Stabilized convex symplectic manifolds are Weinstein}

\author{Yakov Eliashberg$^1$}
\address{$^1$Department of Mathematics\\Stanford University\\Stanford, CA 94305, USA}
\email{eliash@stanford.edu}

\author{Noboru Ogawa$^2$}
\address{$^2$Department of Mathematics\\Tokai University\\4-1-1 Kitakaname, Hiratsuka-shi, Kanagawa, 259-1292 Japan}
\email{nogawa@tsc.u-tokai.ac.jp}

\author{Toru Yoshiyasu$^3$}
\address{$^3$Center for Genomic Medicine, Graduate School of Medicine\\Kyoto University\\53 Shogoinkawahara-cho, Sakyo-ku, Kyoto-City, Kyoto, 606-8507 Japan}
\email{tyoshiyasu@genome.med.kyoto-u.ac.jp}

\begin{abstract}
We show that a stabilized   convex symplectic  (also called Liouville) manifold with the homotopy type of a half dimensional CW-complex is symplectomorphic to a flexible Weinstein manifold.
\end{abstract}

\maketitle

\section{Introduction}\label{intro}

\subsection{Convex symplectic manifolds}\label{sec:convex-sympl}

Recall that a primitive $\lambda$ of a symplectic form $\om$, $d\lambda=\om$, is called a {\em Liouville form}, and the vector field $Z$ which is $\om$-dual to $\lambda$, $\iota(Z)\om=\lambda$, is called a {\em Liouville vector field}.
The equation $\iota(Z)\om=\lambda$ is equivalent to the equation $L_Z\om=\om$, where $L_Z$ is the Lie derivative, i.e.~Liouville vector fields are {\em conformally symplectic}.

An open symplectic manifold $(V,\om)$ with an exact symplectic form $\om$ is called {\em symplectically convex} (see \cite{EG91}), or {\em Liouville} if there exists a Liouville form $\lambda$ such that the corresponding Liouville vector field $Z$ is complete and there exists an exhaustion $\bigcup\limits_{j=1}^\infty V_j$, $V_j\subset V_{j+1}$, by compact domains $V_j$ with smooth boundaries $\p V_j$ such that $Z$ is outward transverse to $\p V_j$.
The domains $V_j$ with this property are called {\em Liouville domains}.

Given a Liouville domain $(V_1,\lambda)$, the attractor $\Core(V_1,\lambda):=\bigcap\limits_{t>0}Z^{-t}(V_1)$ of the field $-Z$ is called the {\em core} of the Liouville domain.
For a convex symplectic manifold $V=\bigcup\limits_{j=1}^\infty V_j$ with a fixed Liouville form $\lambda$, we define its core as $\Core(V,\lambda):=\bigcup\limits_{j=1}^\infty\Core(V_j,\lambda)$.
Equivalently, we can define $\Core(V,\lambda)$ as 
\[
\Core(V,\lambda)=\bigcup_{K\subset V,\; {\textrm{compact}}}\bigcap\limits_{t>0}Z^{-t}(K),
\] 
and this definition shows independence of $\Core(V,\lambda)$ of the choice of exhausting Liouville domains.
Of course, the core does depend on the choice of the Liouville form $\lambda$.

An important class of convex symplectic manifolds is formed by convex symplectic manifolds of finite type, or as they are also called {\em convex symplectic manifolds with cylindrical ends}.
One says that $(V,\om)$ is a convex symplectic manifold of {\em finite type} if it admits a Liouville form $\lambda$ with a Liouville vector field $Z$ and a compact Liouville subdomain $V_1\subset V$, i.e.~a domain with boundary $\p V_1$ transverse to $Z$, such that each point of $V\setminus V_1$ belongs to a $Z$ trajectory originated from a point of $\p V_1$.
The manifold $(V,\om)$ can be identified with the {\em completion} of the Liouville domain $V_1$, i.e.~attaching to $V_1$ the cylindrical end $([0,\infty)\times\p V_1, d(e^s(\lambda|_{\p V_1})))$, where $s$ is the coordinate corresponding to the factor $[0,\infty)$.

Note that the definition of a convex symplectic manifold of finite type fits into the definition of a general convex symplectic manifold by taking the translates $V_n:=Z^{n-1}(V_1)$, $n\geq1,$ as the required exhausting sequence.
For a finite type $(V,\lambda)$ its core is compact:
$\Core(V,\lambda)=\Core(V_1,\lambda)$, and conversely finite type convex symplectic manifolds can be characterized among convex symplectic manifolds as those which have a compact core for some choice of the Liouville form.

Fixing a cylindrical end structure, a contact structure is induced on the ideal boundary $\p_\infty V\cong\p V_1$.
However, this contact boundary is not determined by the symplectomorphism type of $V$.
In fact, as it was shown by Sylvain Courte \cite{Co14}, even the diffeomorphism type of $\p_\infty V$ can depend on the choice of the cylindrical end structure on a given convex symplectic manifold of finite type.

\subsection{Weinstein manifolds}\label{sec:Weinstein}

We say that a convex symplectic manifold $V$ is of {\em Weinstein type}, or simply {\em Weinstein}, if the corresponding Liouville vector field $Z$ can be chosen to admit a Lyapunov function $\phi\colon V\to\R$ which is Morse, or generalized Morse (i.e.~possibly with death-birth singularities).

The Lyapunov condition means that $|d\phi(Z)|\geq c||Z||^2$ for a positive function $c>0$ and some choice of a Riemannian metric on $V$.
We note that the Lyapunov function $\phi$ can always be modified to be {\em exhausting} (i.e.~proper and bounded below) and constant on boundaries $\p V_j$ of domains $V_j$ implied by the definition of symplectic convexity.
The core of a Weinstein manifold $(V,\lambda)$ is stratified by $Z$-stable manifolds of zeroes of $Z$, which are isotropic, see \cite{EG91, CE12}.
Hence, the critical points of a Lyapunov Morse function for a Liouville field have index $\leq n= \frac12 \dim V$, and therefore any Weinstein manifold admits an exhausting Morse function with critical points of index $\leq n$, i.e.~it has {\em Morse type $\leq n$}, and in particular is homotopy equivalent to an $n$-dimensional CW-complex.

Not every convex symplectic manifold is Weinstein.
Indeed, it may have Morse type $>n$.
The first example of this type, a $4$-dimensional convex symplectic manifold of Morse type $3$, was constructed by Dusa McDuff in \cite{Mc91}.
More examples were constructed in \cite{Ge94, Mi95, MNW13}.

The product $(V,\om)\times(V',\om')$ of two symplectically convex manifolds is symplectically convex, and the product of two Weinstein manifolds is Weinstein.
If $(V',\om')=(\R^{2k},\om_{\rm st})$ then the product $(V,\om)\times(\R^{2k},\om_{\rm st})$ is called the {\em stabilization}, or $k$-{\em stabilization}, of $(V,\om)$.

\subsection{Main results}\label{sec:main}

We prove in this paper the following theorem.

\begin{theorem}\label{thm:stab-convex}
Let $V$ be a $(2n-2)$-dimensional convex symplectic manifold of Morse type $\leq n$.
Then its $1$-stabilization $X$ is Weinstein, and moreover if $n\geq3$ flexible Weinstein, {\rm see Section \ref{sec:flexibility} below for the definition and discussion of flexibility}.
In particular, the $1$-stabilization of McDuff's example in \cite{Mc91} is Weinstein.
\end{theorem}

\begin{remark}
	\begin{enumerate}
	\item It was proven in \cite{EG91} that for any two tangentially homotopy equivalent convex symplectic manifolds their 2-stabilizations are symplectomorphic.
	Moreover, it was shown that for Weinstein manifolds $1$-stabilization is sufficient.
	This implies that a $2$-stabilization of a $(2n-2)$-dimensional convex symplectic manifold of Morse type $\leq n$ is Weinstein.
	The improvement in this paper became possible thanks to the development of the theory of flexible Weinstein manifolds, see \cite{CE12}.
	\item Even if $V$ is of finite type, we do not know whether the ideal contact boundary of its stabilization is isomorphic to   the ideal contact boundary of the corresponding flexible Weinstein manifold.
	\end{enumerate}
\end{remark} 

Theorem \ref{thm:stab-convex} is a corollary of a more general theorem which we formulate below.

Given a manifold $V$ and a closed subset $A\subset V$, we say that each point of $A$ {\em has an access to infinity} if every compact subset $B\subset A$ has an arbitrarily small open neighborhood $U\supset B$ such that each connected component $C$ of $V\setminus U$ is not compact.

For instance, if $V$ is a non-compact connected manifold and $A\subset V$ is a closed (as a subset) submanifold of $V$ of codimension $>1$ then each point of $A$ has an access to infinity.
For codimension $1$ connected submanifolds, the condition is violated only for compact submanifolds homological to $0$.

If $V$ is a symplectic manifold and $ A\subset V$ is a locally closed subset, we say that $A$ {\em admits a symplectic extension of positive codimension} if for each point $a\in A$ there exists a neighborhood $U_a\ni a$ in $V$ and a closed (as a subset) symplectic submanifold $\Sigma_a\subset U_a$ of positive codimension such that $U_a\cap A\subset \Sigma_a$, and each point of $U_a\cap A$ has an access to infinity in $\Sigma_a$.

\begin{theorem}\label{thm:general}
Let $(X,\om)$ be a $2n$-dimensional, $n\geq 3$, convex symplectic manifold of Morse type $\leq n$.
Suppose that for an appropriate choice of a Liouville form $\lambda$, its core $C:=\Core(X,\lambda)$ can be presented as a finite or countable union $C=\bigcup\limits_{i\geq1} C_i$ of disjoint sets $C_i$ which admits a symplectic extension of positive codimension and such that $\bigcup\limits_{i\leq j} C_i$ is compact for all $j\geq1$.
Then $X$ is symplectomorphic to a flexible Weinstein manifold.
\end{theorem}

The core of any stabilized convex symplectic manifold clearly admits a symplectic extension of positive codimension, and hence Theorem~\ref{thm:stab-convex} is a special case of Theorem~\ref{thm:general}.

Here is another corollary of Theorem~\ref{thm:general}.
We say that a submanifold $A$ of a symplectic manifold $V$ is {\em nowhere coisotropic} if each tangent plane $T_{x}A\subset T_{x}V$ is not coisotropic, i.e.~$(T_{x}A)^{\perp_{\om}}\nsubseteq T_{x}A$.

\begin{theorem}\label{thm:stratified}
Let $X$ be a $2n$-dimensional, $n\geq 3$, convex symplectic manifold of Morse type $\leq n$.
Suppose that for an appropriate choice of a Liouville form $\lambda$, the core of $X$ admits a stratification $\Core(X,\lambda)=\bigcup_{i\geq1}S_i$ of codimension $\geq 3$ such that each stratum is nowhere coisotropic.
Then the convex symplectic manifold $X$ is Weinstein, and moreover flexible Weinstein.
\end{theorem}

We use above the term stratification in a weak sense.
A {\em stratified} closed subset $A\subset V$ is a closed set presented as a finite or countable union of locally closed submanifolds $A_i$, called {\em strata}, $A:=\bigcup\limits_{i\geq1} A_i$ such that all unions $\bigcup\limits_{i\leq j}A_i$ are compact.

\begin{proof}[Proof of Theorem \ref{thm:stratified}]
It is sufficient to prove that each stratum $S_i$ admits a symplectic extension of positive codimension.
Note that any non-coisotropic subspace $A$ of codimension $\geq 3$ in a symplectic vector space $(B,\om)$ is contained in a symplectic subspace $C\subset B$ such that $\dim A<\dim C<\dim B$.
Indeed, take a vector $v\in A^{\perp_\om}\setminus A$ and consider its $\om$-orthogonal complement subspace $v^{\perp_\om}\subset B$.
Note that $A\subset v^{\perp_\om}$.
Then any codimension $1$ subspace $C\subset v^{\perp_\om}$ which is transverse to $v$ and contains $A$ is a codimension $2$ symplectic subspace of $B$.
Given $x\in S_i$, we therefore can find a $(2n-2)$-dimensional symplectic subspace $C_x$ so that $T_{x}S_i\subset C_x\subset T_{x}X$.
Let us choose complementary subspace $ \theta_x\subset C_x$ such that $ \theta_x\oplus T_{x}S_i=C_x$ and extend it continuously to a field $\theta$ of planes transverse to $S_i$ on a neighborhood $U_x$ of $x$ in $S_i$.
If the neighborhood $U_x$ is small enough then the space $ C_y:=\Span(\theta_y,T_{y}S_i)$ is symplectic for each $y\in U_x$, and so is the codimension $2$ symplectic hypersurface $\Sigma$ containing $U_x$ and tangent to the plane field $\theta$.
\end{proof}

\subsection*{Plan of the paper}

In Sections~\ref{sec:flexibility} and \ref{sec:opensympemb}, we review Weinstein flexibility~\cite{CE12, EM13, ELM20} and Gromov's $h$-principle for exact symplectic embeddings of open symplectic manifolds~\cite{Gr86, EM02}, respectively.
Section~\ref{sec:homotopy-symplectomorphism} is a remark on the existence of a Liouville homotopy and a symplectomorphism.
With the help of the above tools we first prove Theorem~\ref{thm:stab-convex} in Section~\ref{sec:proof1} to make main ideas more transparent, and then prove Theorem~\ref{thm:general} in Section~\ref{sec:proof2}.

\subsection*{Acknowledgement}

The collaboration on this paper started during the Kenji Fukaya's 60th birthday conference ``Geometry and Everything'' at the Research Institute for Mathematical Sciences (RIMS), Kyoto University, and continued during the visit of the first author to RIMS.
The authors thank the Institute for its hospitality.
The first author is partially supported by the NSF grant DMS--1807270.
The second author is partially supported by JSPS KAKENHI Grant Number JP17K05283.

\section{Recollection of Weinstein flexibility}\label{sec:flexibility}

\subsection{Loose Legendrian knots}

We recall that an $(n-1)$-dimensional submanifold $\Lambda$ of a $(2n-1)$-dimensional contact manifold $(M,\xi)$ is called {\em Legendrian} if it is tangent to $\xi$.

The contact plane field $\xi$ carries a canonical conformal symplectic structure and tangent planes to a Legendrian are Lagrangian subspaces of $\xi$ for that conformal structure.
A {formal Legendrian submanifold} $\Lambda\to M$ is an $(n-1)$-dimensional smooth submanifold together with a homotopy of its tangent planes to a field of Lagrangian subspaces of $\xi$.
 
In \cite{Mu12}, Emmy Murphy introduced a class of the so-called {\em loose} Legendrian submanifolds in contact manifolds of dimension $\geq 5$ for which a formal Legendrian isotopy between Legendrian embeddings yields a genuine Legendrian isotopy.
We define this class below.
 
We begin with an operation of {\em stabilization} of a Legendrian submanifold which was first introduced in \cite{El90}, see also \cite{Mu12, CE12}.

In $\R^{2n-1}$, $n\geq3$, with the standard contact form $\alpha=dz-\sum\limits_{ i=1}^{n-1} y_i\,dx_i$ consider a Legendrian submanifold $\Lambda_0$ with the front $F_0=\{z^2=x_1^3\}$, i.e.~$\Lambda_0=\{z^2=x_1^3, 4y_1^2=9x_1, y_2=\cdots=y_{n-1}=0\}$.
Let $\R^{n-1}=\{z=0, y=0\}$ be the $x$-coordinate subspace and $\R_+^{n-1}=\{x_1>0\}\cap\R^{n-1}$.
Choose open domains $U$ and $U'\Subset U$ with smooth boundaries and let $\theta\colon\R_+^{n-1}\to\R$ be a function supported in $U$ such that $U'=\{x\in\R_+^{n-1}\mid\theta(x)>2x_1^{\frac32}\}$.
Let $\Lambda_U\subset\R^{2n-1}$ be a Legendrian submanifold whose front is obtained from $F_0$ by replacing the branch $z=-x_1^{\frac32}$ by the graph $z=-x_1^{\frac32}+\theta(x)$.
The Legendrian submanifold $\Lambda_U$ is called the {\em $U$-stabilization} of $\Lambda_0$.
Given any Legendrian submanifold $\Lambda$ in a contact manifold $(M,\xi)$, one can find Darboux coordinates in an arbitrarily small neighborhood $U$ of a point $a\in\Lambda$ such that the pair $(U,\Lambda\cap U)$ is contactomorphic to $(\R^{2n-1},\Lambda_0)$.
Hence the $U$-stabilization operation can be performed on $\Lambda$ in a neighborhood $U\ni a$.
We will keep the notation $\Lambda_U$ for the stabilized Legendrian.
As it was shown in \cite{El90}, the Legendrians $\Lambda$ and $\Lambda_U$ are always smoothly isotopic, and if $\chi(U)=0$ then they are {\em formally Legendrian isotopic}.

A connected Legendrian submanifold $\Lambda\subset M$ is called {\em loose} if it can be {\em destabilized}, i.e.~it is Legendrian isotopic to a stabilization of another Legendrian knot.
A possibly disconnected Legendrian is called loose if each its component is loose in the complement of the others.

\begin{theorem}[Murphy~\cite{Mu12}]\label{thm:Murphy}
For any contact manifold of dimension $\geq 5$, the inclusion of the space of loose Legendrian embeddings into the space of formal Legendrian embeddings is a homotopy equivalence.
\end{theorem}

\subsection{Flexible Weinstein manifolds}

The notion of {\em Weinstein flexibility}, introduced in \cite{CE12}, is based on the theory of loose Legendrians.

Let $(V,\lambda,\phi)$ be a $2n$-dimensional Weinstein manifold.
Consider its partition into {\em elementary cobordisms}:
$V=W_1\cup\cdots\cup W_m\cup\cdots$, where $W_i=\phi^{-1}([c_{i-1},c_i])$ for regular values $c_i$ of $\phi$ separating the critical values $a_i$ of $\phi$, i.e.~$c_0<a_0<c_1<a_1<c_2<\cdots$.
Each cobordism $W_i$ deformation retracts onto the union of the stable disks (with respect to the Liouville vector field $Z$) of its critical points and the stable disk of an index $k$ critical point of value $a_i$ intersects the level set $M_i=\phi^{-1}(c_i)$ in the $(k-1)$-dimensional isotropic {\em attaching sphere} for the contact structure $\{\lambda|_{M_i}=0\}$.

The Weinstein structure $(V,\lambda,\phi)$ is called {\em flexible} if for each cobordism $W_i$ the attaching Legendrian spheres of critical points of index $n$ on the level $a_i$ form a loose Legendrian link in the contact level set $M_i$.
In particular, subcritical Weinstein manifolds (i.e.~those for which $\phi$ has no critical points of index $n$) are flexible.

The following $h$-principle type result clarifies the term ``flexible".

\begin{theorem}[Cieliebak--Eliashberg~\cite{CE12}]\label{thm:flex}
For Weinstein structures on a fixed manifold or domain $V$ of dimension $2n\geq 6$, the following statements hold.
	\begin{enumerate}
	\item (Existence)
	Given a non-degenerate $2$-form $\eta$ and an exhausting Morse function $\phi\colon V\to\R$ without critical points of index $>n$, there exists a flexible Weinstein structure $(\lambda,\phi)$ (with the same function $\phi$) such that $\eta$ and $d\lambda$ are homotopic as non-degenerate $2$-forms.\label{thm:flex_exist}
	\item (Homotopy)
	Two flexible Weinstein structures $(\lambda_0,\phi_0)$ and $(\lambda_1,\phi_1)$ are Weinstein homotopic if and only if $d\lambda_0$ and $d\lambda_1$ are homotopic as non-degenerate $2$-forms.
	\item (Morse--Smale theory for Lyapunov functions)
	Given a flexible Weinstein structure $(\lambda,\phi)$ and any Morse function $\psi\colon V\to\R$ without critical points of index $>n$, there exists a Weinstein homotopy $(\lambda_t,\phi_t)$ with $(\lambda_0,\phi_0)=(\lambda,\phi)$ and $\phi_1=\psi$.
	\end{enumerate}
\end{theorem}

The definition of flexibility naturally extends to {\em Weinstein cobordisms}.

It is important to point out that the flexibility property is not invariant under Weinstein homotopy, see \cite{MS18}.
When calling a symplectic manifold flexible Weinstein, we always mean the existence of a flexible Weinstein structure for the given symplectic form.

\subsection{Symplectic embeddings of flexible Weinstein manifolds}\label{subsec:W-emb}

For two symplectic manifolds $(W,\om)$ and $(X,\eta)$, a {\em formal symplectic embedding} of $W$ into $X$ is a smooth embedding $f\colon W\to X$ together with a homotopy $\Phi_t\colon TW\to f^{\ast}TX$, $t\in[0,1],$ of injective bundle homomorphisms such that $\Phi_0=df$ and $\Phi_{1}^{\ast}\eta=\om$.
Any genuine symplectic embedding $f\colon X\to W$ can be considered formal by setting $\Phi_t\equiv df$.

A symplectic embedding $f\colon(W,d\lambda)\to(X,d\mu)$ between two exact symplectic manifolds with fixed Liouville forms is called {\em exact} if $f^{\ast}\mu=\lambda+dH$ for some function $H$ on $W$.
Note that if $W$ is compact then given an exact symplectic isotopy $f_t\colon(W,d\lambda)\to(X,d\mu)$ there exists an ambient Hamiltonian isotopy $F_t\colon X\to X$ such that $F_t|_{f_0(W)}=f_t$, $t\in[0,1]$.
Hence, we will refer in this paper to an exact symplectic isotopy as a {\em Hamiltonian} isotopy.

\begin{theorem}[Eliashberg--Murphy~\cite{EM13}, Eliashberg--Lazarev--Murphy~\cite{ELM20}]\label{thm:W-emb}
Let $(W,\lambda,\phi)$ be a $2n$-dimensional Weinstein domain and $W_0$ its Weinstein subdomain.
Suppose that the Weinstein cobordism $(W\setminus\Int W_0,\lambda,\phi)$ is flexible.
Let $(X,d\mu)$ be a convex symplectic manifold of the same dimension $2n$ such that the Liouville vector field $Z$ dual to $\mu$ is forward complete.
Then
	\begin{enumerate}
	\item Any formal symplectic embedding $f\colon(W,d\lambda)\to(X,d\mu)$ which is a genuine exact symplectic embedding on $W_0$ is formally isotopic rel.~$W_0$ to a genuine exact symplectic embedding.\label{thm:W-emb_non-para}
	\item Any two exact symplectic embeddings $f_0,f_1\colon(W,d\lambda)\to(X,d\mu)$ which coincide on $W_0$ and are formally isotopic rel.~$W_0$ can be connected by a Hamiltonian isotopy $f_t\colon(W,d\lambda)\to(X,d\mu)$, $t\in[0,1]$, fixed on $W_0$.\label{thm:W-emb_1-para}
	\end{enumerate}
\end{theorem}

The non-parametric part (\ref{thm:W-emb_non-para}) is proven in \cite{EM13}.
The parametric part (\ref{thm:W-emb_1-para}) will appear in \cite{ELM20}.

\section{h-principle for symplectic embeddings}\label{sec:opensympemb}

In this section, we review Gromov's $h$-principle for symplectic embeddings of open symplectic manifolds~\cite{Gr86}, see also \cite{EM02}.
We continue to use in this paper the term ``symplectic" rather than {\em isometric} as in \cite[3.4.2 (B)]{Gr86} or {\em isosymplectic} as in \cite[12.1.1]{EM02}.

Recall that given a manifold $V$ and a closed subset $A\subset V$ we say that each point of $A$ {\em has an access to infinity} if every compact subset $B\subset A$ has an arbitrarily small open neighborhood $U\supset B$ such that each connected component $C$ of $V\setminus U$ is not compact.
Slightly reformulating Gromov's $h$-principle for symplectic embeddings from \cite{Gr86}, we have the following theorem.
We use below Gromov's notation $\Op A$ for an unspecified neighborhood of a closed subset $A$.

\begin{theorem}[Gromov~\cite{Gr86}, see also ~\cite{EM02}]\label{thm:X-emb}
Let $(W,\om)$ and $(X,\eta)$ be symplectic manifolds of dimension $2n$ and $2m$, respectively.
Suppose that $X$ is an open manifold and $m<n$.
Then the following statements hold.
	\begin{enumerate}
	\item For a formal symplectic embedding $(\varphi, \Phi_t)$ of $(X, \eta)$ into $(W, \om)$, there exists a symplectic embedding $f\colon(X, \eta)\to(W, \om)$ formally isotopic to $(\varphi, \Phi_t)$.\label{thm:X-emb_non-para}
	\item Any symplectic embeddings $f_0, f_1\colon(X, \eta)\to(W, \om)$ which are formally isotopic can be connected by a symplectic isotopy $f_t\colon(X, \eta)\to(W, \om)$, $t\in[0,1]$.\label{thm:X-emb_1-para}
	\item Let $A\subset X$ be a closed subset such that each its point has an access to infinity and $(\varphi,\Phi_t)$ a formal symplectic embedding of $(X,\eta)$ into $(W,\om)$ which is a genuine symplectic embedding on $\Op A$.
	Then, there exists a symplectic embedding $f\colon(X, \eta)\to(W, \om)$ formally isotopic to $(\varphi, \Phi_t)$ rel.~$A$.\label{thm:X-emb_rel}
	\item Let $A$ be as in (\ref{thm:X-emb_rel}).
	Then for any two symplectic embeddings $ f_0,f_1\colon X\to W$ which coincide on $\Op A$ and are formally isotopic rel.~$A$, there exists a symplectic isotopy $f_t\colon(X,\eta)\to(W,\om)$, $t\in[0,1]$, fixed on $\Op A$.\label{thm:X-emb_rel_para}
	\end{enumerate}
\end{theorem}

If symplectic forms $\om=d\lambda$ and $\eta=d\mu$ are exact, then one can talk about {\em exact} symplectic embeddings $f\colon X\to W$ which satisfy the condition $f^{\ast}\lambda=\mu+dH$, see Section~\ref{subsec:W-emb} above.

\begin{proposition}\label{prop:to-exact}
Let $(X,\eta=d\mu)$ be a convex symplectic manifold.
Then for any symplectomorphism $f_0\colon(X,\eta)\to(X,\eta)$, there exists a symplectic diffeotopy $f_t\colon(X,\eta)\to(X,\eta)$ such that $f_1$ is exact, i.e.~$f^{\ast}_{1}\mu=\mu+dH$ for a smooth function $H\colon X\to\R$.
\end{proposition}

For the case when $(X,\eta)$ is a finite type convex symplectic manifold, this was proven in \cite[Lemma 11.2]{CE12}.
To prove the statement in the general case we will need the following two lemmas.

\begin{lemma}\label{lm:dual-to-closed}
Let $((0,\infty)\times\Sigma,\om=d(s\alpha))$ be the symplectization of a contact manifold $(\Sigma,\xi=\ker\alpha)$ with a fixed contact form $\alpha$.
Let $R$ be the Reeb vector field of $\alpha$ (i.e.~$\iota(R)d\alpha=0$ and $\alpha(R)=1$).
For a closed $1$-form $\theta$ on $\Sigma$, denote $\wh\theta:=\pi^{\ast}\theta$ and let $Y$ be the symplectic vector field $\om$-dual to $\wh\theta$, i.e.~$\iota(Y)\om=\wh\theta$.
Here $\pi\colon(0,\infty)\times\Sigma\to\Sigma$ is the projection to the second factor.
Then one has the equality $ds(Y) =\theta(R)$.
\end{lemma}

\begin{proof}
Let us write $Y=aR+b\frac{\p}{\p s}+Y_\xi$, where $Y_\xi\in\xi$ and $a,b\colon(0,\infty)\times\Sigma\to\R$.
Then
\[
\theta(R)=(\iota(Y)(ds\wedge\alpha+sd\alpha))(R)=ds\wedge\alpha(aR+b\frac{\p}{\p s}+Y_\xi,R)+sd\alpha(Y,R)=b=ds(Y).
\]
\end{proof}

\begin{lemma}\label{lm:good-exhaustion}
Let $(X,d\mu)$ be a convex symplectic manifold and $Z$ the Liouville field corresponding to $\mu$.
Then there exists an exhaustion $X=\bigcup\limits_{j=1}^\infty X_j$ by compact domains with smooth boundaries transverse to $Z$ such that the following condition is satisfied: for any $j\geq 1$ one has
	\begin{align}\label{eq:disjunction}
	Z^{\ln(1+2T_j)}(\p X_j)\subset\Int X_{j+1},\quad j\geq1,
	\end{align}
where $T_j:=\max\{1,\mathop{\max}\limits_{x\in\p X_j}|\theta(R_j(x))|\}$ and $R_j$ is the Reeb vector field of the contact form $\mu|_{\p X_j}$.
\end{lemma}

\begin{proof}
We begin with any exhaustion $\bigcup\limits_{j=1}^\infty X^0_j$ by compact domains with smooth boundaries transverse to $Z$ and then inductively modify it to ensure the property \eqref{eq:disjunction} by using the following procedure.
Set $X_1:=X_1^0$ and $T_1:=\max\{1,\mathop{\max}\limits_{x\in\p X_1}\left|\theta(R_1(x))\right|\}$, where $R_1$ is the Reeb vector field of the contact form $\mu|_{\p X_1}$.
Define $X^1_j:=Z^{\ln(1+3T_1)}(X^0_j)$, $j\geq 2$, and denote $X_2:=X_2^1$ and $T_2:= \max\{1,\mathop{\max}\limits_{x\in\p X_2}\left|\theta(R_2(x))\right|\}$, where $R_2$ is the Reeb vector field of the contact form $\mu|_{\p X_2}$.
Define $X^2_j:=Z^{\ln(1+3T_2)}(X^1_j)$, $j\geq 3$, and denote $X_3:=X_3^2$.
Continuing this process we construct the required exhaustion.
\end{proof}

\begin{proof}[Proof of Proposition \ref{prop:to-exact}]
We have $f_{0}^{\ast}\mu=\mu-\theta$ for a closed $1$-form $\theta$.
Let $Z$ be the Liouville field corresponding to the Liouville form $\mu$.
Choose an exhaustion $X=\bigcup\limits_{j=1}^\infty X_j$ which satisfies the property \eqref{eq:disjunction}.

Consider disjoint domains
\[
U_j:=\bigcup\limits_{t\in[0,\ln(1+2T_j)]}Z^t(\p X_j)\subset X
\]
and set $\alpha_j:=\mu|_{\p X_j}$, $j\geq1$.
These domains can be identified with the domains $[1,1+2T_j]\times\p X_j$ in the symplectizations $((0,\infty)\times\p X_j, d(s\alpha_j))$ of the contact manifolds $(\p X_j,\ker\alpha_j)$ via Liouville isomorphisms
\[
\phi_j\colon( [1,1+2T_j]\times\p X_j, s\alpha_j)\to(U_j, \mu)\ :\ (s,x)\mapsto Z^{\ln s}(x),
\]
where $s\in[1,1+2T_j]$ and $x\in\p X_j$.
Let $\theta_j$ be the closed 1-form on $U_j$ defined by the formula
\[
\theta_j:=(\phi_j)_{\ast}\pi_j^{\ast}(\theta|_{\p X_j}),
\]
where $\pi_j\colon [1,1+2T_j]\times\p X_j\to\p X_j$ is the projection to the second factor.
Note that $\theta|_{U_j}-\theta_j=dH_j$ for a smooth function $H_j\colon U_j\to\R$.
Let $\delta_j\colon U_j\to\R$ be a cut-off function supported in $U_j$ and equal to $1$ on $\wh U_j:=\phi_j([1+T_j/2,1+3T_j/2]\times\p X_j)$.
Then the closed $1$-form
\[
\wh\theta=\theta-dG,\quad\hbox{where }G=\left(\sum\limits_{ j=1}^{\infty}\delta_{j}H_{j}\right),
\]
coincides with $\theta_j$ on $\wh U_j$ for all $j\geq 1$.
We claim that the symplectic vector field $Y$ which is $\eta$-dual to $\wh\theta$ is complete, i.e.~its flow $Y^t$ is defined for all time $t\in\R$.
Indeed, according to Lemma~\ref{lm:dual-to-closed} any trajectory entering $U_i$ spends there the time $\geq 1$, and hence in time $\leq T$ it can cross only finitely many domains $U_j$.
Note that $L_{Y}\wh\theta=\iota(Y)d\wh\theta+d(\iota(Y)\wh\theta)=0$ and $L_{Y}\mu=\iota(Y) \eta +d(\mu(Y))=\wh\theta+ dH'$, where $H':=\mu(Y)$.
Hence, by defining $H'_t:=\int_{0}^{t}(H'\circ Y^s)\,ds$, we compute $(Y^t)^{\ast}\wh\theta=\wh\theta$ and $(Y^t)^{\ast} \mu=\mu+t\wh\theta+dH'_t$.
We can now define the required isotopy by the formula $f_t:=f_0\circ Y^{t}$.
Then
	\begin{align*}
	f_{1}^{\ast}\mu & =(Y^{1})^{\ast}f_{0}^{\ast}\mu=(Y^1)^{\ast}(\mu-\theta)=(Y^1)^{\ast}(\mu-\wh\theta-dG)\\
	& =\mu+ \wh\theta+dH'_1- \wh\theta-d(G\circ Y^1) =\mu+ d(H'_1-G\circ Y^1).
	\end{align*}
\end{proof}
 
Hence, if in Theorem~\ref{thm:X-emb} we assume that $\om=d\lambda$, $\eta=d\mu$, and $(X,d\mu)$ is symplectically convex, then we can arrange the constructed symplectic embeddings in (\ref{thm:X-emb_non-para}) and (\ref{thm:X-emb_rel}) be exact, and the symplectic isotopies in (\ref{thm:X-emb_1-para}) and (\ref{thm:X-emb_rel_para}) be Hamiltonian.

\section{Liouville homotopy vs symplectomorphism}\label{sec:homotopy-symplectomorphism}

The following notion of Liouville homotopy, which formalizes the concept of a smooth family of convex symplectic structures on a given manifold, was introduced in \cite{CE12}.
A smooth family $\mu_s$, $s\in[0,1]$, of Liouville forms on a manifold $X$ is called a {\it simple Liouville homotopy} if there exists a smooth family of exhaustions $X=\bigcup_{k=1}^{\infty}X_s^k$ by compact domains $X_s^k\subset X$ with smooth boundaries along which the corresponding Liouville field $Z_s$ is outward pointing.
A {\em Liouville homotopy} is a composition of finitely many simple Liouville homotopies.
It was shown in \cite[Proposition 11.8]{CE12} that given a Liouville homotopy $\mu_s$ there exists an isotopy $\varphi_s\colon X\to X$, $s\in[0,1]$, starting from $\varphi_0=\mathrm{id}_X$ such that $\varphi_{s}^{\ast}\mu_s=\mu_0+dH_s$, and in particular the forms $\wh\mu_s:=\varphi_s^*\mu_s$ are Liouville for the same symplectic structure $\om=d\mu_0$.

The following proposition shows that  the converse is also true.

\begin{proposition}\label{prop:symp-hom}
Let $(X,\om)$ and $(X',\om')$ be two symplectomorphic convex symplectic manifolds.
Then there exist a symplectomorphism $\varphi\colon(X,\om)\to(X',\om')$ and a Liouville homotopy $\mu_s$ connecting $\mu_0=\mu$ and $\mu_1=\varphi^{\ast}\mu'$.
\end{proposition}

\begin{proof}
According to Proposition~\ref{prop:to-exact} the symplectomorphism $\varphi$ can be chosen to satisfy $\varphi^{\ast}\mu'=\mu+dH$ for a smooth function $H$ on $X$.
Choose the exhaustions $X=\bigcup\limits_{j=1}^\infty X_j^0$ and $X=\bigcup\limits_{j=1}^\infty X_j^1$ defining convex structures for forms $\mu_0:=\mu$ and $\mu_1:=\mu+dH$, respectively.
We can arrange that
\[
X_j^1\subset \Int X_j^0\subset X_j^0\subset\Int X_{j+1}^{1}
\]
for all $j\geq1$.
Let $\wt H$ be a smooth function which is equal to $0$ on $\Op\left(\bigcup\limits_{j=1}^{\infty}\p X_j^0\right)$ and equal to $H$ on 
$\Op\left(\bigcup\limits_{j=1}^{\infty}\p X_j^1\right)$.
Then the required Liouville homotopy can be defined as the composition of two simple Liouville homotopies:
\[
\mu_s:=
	\begin{cases}
	\mu_0+2sd\wt H, & s\in[0,1/2];\cr
	\mu_0+d\wt H+(2s-1)d(H-\wt H), & s\in[1/2,1].
	\end{cases}
\]
with the constant exhaustions $X=\bigcup\limits_{j=1}^\infty X_j^0$ and $X=\bigcup\limits_{j=1}^\infty X_j^1$, respectively.
\end{proof}

The notion of {\em Weinstein homotopy} can be defined in a similar way.
However, it is unknown whether two Weinstein structures on the same symplectic manifold are homotopic.

\section{Proof of Theorem~\ref{thm:stab-convex}}\label{sec:proof1}

If $n=2$ then $\dim V=2$.
Any $2$-dimensional convex symplectic manifold is Weinstein and the theorem is trivially true.
Hence, we assume that $n\geq 3$.

Choose a Liouville form $\mu_V$ on the convex symplectic manifold $V$ and denote by $\mu$ the corresponding stabilized Liouville form $\mu_V+\frac12(xdy-ydx)$ on $X=V\times\R^2$.
Denote $\eta:=d\mu$.
By the assumption, $X$ is of Morse type $\leq n$.
Take an exhausting Morse function $\phi\colon X\to\R$ without critical points of index $>n$.
Applying Theorem~\ref{thm:flex}~(\ref{thm:flex_exist}) to the pair $(\eta,\phi)$, we obtain a flexible Weinstein structure $\mathfrak{W}=(\om=d\lambda,\phi)$ on $X$ such that symplectic forms $\eta$ and $\om$ are homotopic as non-degenerate $2$-forms.
For the sake of convenience, the ambient space of $\mathfrak{W}$ is denoted by $W$ instead of $X$.
Thus there exists a pair $(\varphi,\Phi_{s})$ where $\varphi\colon X\to W$ is the identity and $\Phi_{s}\colon TX\to TW$, $s\in[0,1]$, is a homotopy of bundle isomorphisms covering $\varphi$ starting at $\Phi_{0}=d\varphi$ and ending at a symplectic isomorphism $\Phi_{1}=\Phi\colon(TX,\eta)\to(TW,\om)$.

The goal of this section is to construct an exact symplectomorphism $F\colon(X,\mu)\to(W,\lambda)$.
This will be given by the telescope construction, the so-called {\em Mazur trick}, see \cite{Ma61}, following the scheme of the proof in \cite[Proposition 2.2.A]{EG91}.
Take exhaustions $X_{1}\subset X_{2}\subset\cdots\subset X$, $\bigcup\limits_{i=1}^{\infty}X_{i}=X$, by Liouville subdomains of $X$ and $W_{1}\subset W_{2}\subset\cdots\subset W$, $\bigcup\limits_{i=1}^{\infty}W_{i}=W$, by Weinstein subdomains of $W$.
 
Set $\mu_{i}=\mu|_{X_{i}}$ and $\lambda_{i}=\lambda|_{W_{i}}$ for $i\geq1$.
The construction is split into several steps.

\begin{lemma}\label{lem:LW}
For each $i\geq 1$ there exists an exact symplectic embedding $f_{i}\colon(X_{i},\mu_{i})\to(W,\lambda)$ 
which is formally isotopic to $(\varphi|_{X_{i}},\Phi|_{TX_{i}})$.
\end{lemma}

\begin{proof}
Since $V$ is an open symplectic manifold of dimension $2n-2<\dim W$, we can apply Theorem~\ref{thm:X-emb}~(\ref{thm:X-emb_non-para}) to $(\varphi|_V,\Phi_{s}|_{TV})$ and obtain an exact symplectic embedding $V=V\times\{0\}\to W$.
Moreover, it extends to an open neighborhood $U$ of $V$, and hence we get an exact symplectic embedding $h_{U}\colon(U,\mu|_{U})\to(W,\lambda)$ which is formally isotopic to $(\varphi|_{U},\Phi|_{TU})$.

We have ${\Core(X_{i},\mu_i)}\subset V\subset U$, and therefore there exists $t_{i}>0$ such that $Z^{-t_{i}}_{\mu}(X_{i})\subset U$, where $Z_{\mu}^{t}$ stands for the flow generated by the Liouville vector field of $\mu$.
Set $h_{i}:=h_{U}|_{Z^{-t_{i}}_{\mu}(X_{i})}$.
Using the flow $Z_\lambda^t$ of the Liouville field $Z_\lambda$ we construct an exact symplectic embedding $f_i\colon(X_{i},\mu_{i})\to(W,\lambda)$ by the formula
\[
f_{i}=Z^{t_{i}}_{\lambda}\circ h_{i}\circ Z^{-t_{i}}_{\mu}.
\]
Indeed,
	\begin{align*}
	f_{i}^{\ast}\lambda & =(Z^{-t_{i}}_{\mu})^{\ast}\circ h_{i}^{\ast}\circ(Z^{t_{i}}_{\lambda})^{\ast}(\lambda)=(Z^{-t_{i}}_{\mu})^{\ast}\circ h_{i}^{\ast}(e^{t_{i}}\lambda)= e^{t_{i}}(Z^{-t_{i}}_{\mu})^{\ast}(\mu_{i}+dH)\\
	& = e^{t_{i}}(e^{-t_{i}}\mu_{i}+d(H\circ Z^{-t_{i}}_{\mu}))=\mu_{i}+d(e^{t_{i}}(H\circ Z^{-t_{i}}_{\mu})).
	\end{align*}

By the construction, $f_{i} $ is formally isotopic to $(\varphi|_{X_{i}},\Phi|_{TX_{i}})$.
\end{proof}

The next lemma is a special case of Theorem~\ref{thm:W-emb}~(\ref{thm:W-emb_non-para}).

\begin{lemma}\label{lem:WL}
There exists an exact symplectic embedding $g_i\colon(W_i,\lambda_{i})\to(X,\mu)$ which is formally isotopic to $(\varphi^{-1}|_{W_{i}},\Phi^{-1}|_{TW_{i}})$.
\end{lemma}

There exist subfamilies $\{X_{i_k}\}$ and $\{W_{j_k}\}$ such that $f_{i_k}(X_{i_k})\subset W_{j_k}$, $\varphi(X_{i_k})\subset W_{j_k}$, $g_{j_k}(W_{j_k})\subset X_{i_{k+1}}$, and $\varphi^{-1}(W_{j_k})\subset X_{i_{k+1}}$.
After renumbering, we have the following diagram:
\[
\xymatrix{
(X_{1}, \mu_{1}) ~\ar@{^{(}-{>}}[r]^{\iota_{X_{1}}} \ar[d]^{f_{1}} &~ (X_{2},\mu_{2}) ~\ar@{^{(}-{>}}[r]^{\iota_{X_{2}}} \ar[d]^{f_{2}} &~ (X_{3},\mu_{3}) ~\ar[d]^{f_{3}} \ar@{^{(}-{>}}[r]^{\iota_{X_{3}}} & ~~~\cdots \ar[d] \cdots\cdots ~\ar@{^{(}-{>}}[r] & (X,\mu)\\
~ (W_{1}, \lambda_{1}) ~\ar@{^{(}-{>}}[r]_{\iota_{W_{1}}} \ar[ru]_{g_{1}} &~ (W_{2}, \lambda_{2}) ~\ar@{^{(}-{>}}[r]_{\iota_{W_{2}}} \ar[ru]_{g_{2}} &~ (W_{3},\lambda_{3}) ~\ar@{^{(}-{>}}[r]_{\iota_{W_{3}}} \ar[ru]_{g_{3}} & ~~~\cdots\cdots\cdots ~\ar@{^{(}-{>}}[r] & (W,\lambda),
}
\]
where $\iota_{X_k}$ and $\iota_{W_k}$ are the inclusions.

\begin{lemma}\label{lem:ham}
The compositions $g_{k}\circ f_{k}\colon X_{k}\to X_{k+1}$ and $f_{k+1}\circ g_{k}\colon W_{k}\to W_{k+1}$ are Hamiltonian isotopic to the inclusions $\iota_{X_{k}}\colon X_{k}\to X_{k+1}$ and $\iota_{W_{k}}\colon W_{k}\to W_{k+1}$, respectively.
\end{lemma}

\begin{proof}
By Lemmas~\ref{lem:LW} and \ref{lem:WL}, $g_{k}\circ f_{k}$ is formally isotopic to $\iota_{X_{k}}=\varphi^{-1}|_{W_{k}}\circ \varphi|_{X_{k}}$.
Applying Theorem~\ref{thm:X-emb}~(\ref{thm:X-emb_1-para}) to this formal isotopy restricted on $V_{k}=V\cap X_{k}$, we get a Hamiltonian isotopy $h^{s}_{k} \colon V_{k} \to X_{k+1}$, $s\in [0,1]$, such that $h^{0}_{k}=\iota_{X_{k}}|_{V_{k}}$ and $h^{1}_{k}=g_{k}\circ f_{k}|_{V_{k}}$.
Arguing as in the proof of Lemma~\ref{lem:LW}, we can define a Hamiltonian isotopy $\psi^{s}_{k}\colon X_{k}\to X_{k+1}$ between $\psi^{0}_{k}=\iota_{X_{k}}$ and $\psi^{1}_{k} =g_{k}\circ f_{k}$ by the formula
\[
\psi^{s}_{k}:=Z^{t_{k}}_{\mu}\circ\wt{h}^{s}_{k}\circ Z^{-t_{k}}_{\mu}.
\]
Here $\wt{h}^{s}_{k}$ is an extension of $h^{s}_{k}$ to an open neighborhood $U_{k}$ of $V_{k}$ as in the proof of Lemma~\ref{lem:LW} and $t_{k}$ is a sufficiently large number so that $Z^{-t_{k}}_{\mu}(X_{k})\subset U_{k}$.
	 
Similarly, we use Theorem~\ref{thm:W-emb}~(\ref{thm:W-emb_1-para}) to construct a Hamiltonian isotopy connecting $f_{k+1}\circ g_{k}$ and $\iota_{W_{k}}$.
\end{proof} 

\begin{proof}[Proof of Theorem~\ref{thm:stab-convex}]
We construct an exact symplectomorphism from $X$ to $W$ by induction over $k$.
First, set $F_{1}:=f_{1}\colon(X_{1},\mu_{1})\to(W_{1},\lambda_{1})$.
Lemma~\ref{lem:ham} constructs for any $k\geq1$ a Hamiltonian isotopy connecting the inclusion $\iota_{X_k}\colon X_k\to X_{k+1}$ with the composition $g_k\circ f_k$.
Hence cutting off this isotopy outside $X_k$ we get a Hamiltonian isotopy $G^{s}_{k}\colon(X,\mu)\to(X,\mu)$, $s\in[0,1]$, such that $G^{0}_{k}=\mathrm{id}_{X}$, $G^{1}_{k}|_{X_{k}}=g_{k}\circ f_{k}$, and $\mathrm{supp}(G^{s}_{k})\subset X_{k+1}$.
Similarly, Lemma~\ref{lem:ham} allows us to construct a Hamiltonian isotopy $H^{s}_{k}\colon(W,\lambda)\to(W,\lambda)$, $s\in[0,1]$, such that $H^{0}_{k}=\mathrm{id}_{W}$, $H^{1}_{k}|_{W_{k}}=f_{k+1}\circ g_{k}$, and $\mathrm{supp}(H^{s}_{k})\subset W_{k+1}$.
Set $G_{k}:=G^{1}_{k}$ and $H_{k}:=H^{1}_{k}$.
For $k\geq 2$, we define the exact symplectic embedding 
\[
F_{k}:=(H_{k-1}\circ\cdots\circ H_{1})^{-1}\circ f_{k}\circ (G_{k-1}\circ\cdots\circ G_{1})|_{X_{k}}\colon(X_{k},\mu_{k})\to(W_{k},\lambda_{k}).
\]
Since $\mathrm{supp}(G_{j})\subset X_{j+1}$ and $H^{-1}_{k-1}\circ f_{k}\circ G_{k-1}|_{X_{k-1}}=f_{k-1}$, the restriction $F_{k}|_{X_{k-1}}$ is equal to $F_{k-1}$, $k\geq 2$.
Hence, we can define an exact symplectic embedding $F\colon X\to W$ by setting $F:=F_{k}$ on $X_{k}$ for $k\geq 1$.
Let us show that $F$ is surjective.
Note that $W_{k-1}\subset F_{k}(X_{k})$ for $k\geq 2$.
Indeed, we have
	\begin{align*}
	H_{k-1}\circ H_{k-2}\circ\cdots\circ H_{1} (W_{k-1}) & =H_{k-1} (W_{k-1})\\
	& =f_{k}(g_{k-1}(W_{k-1}))\\
	& \subset f_{k}(X_{k})\\
	& =f_{k}(G_{k-1}\circ\cdots\circ G_{1}(X_{k})),
	\end{align*}
and thus $W_{k-1}\subset F_{k}(X_{k})$.
Hence,
\[
W=\bigcup_{k\geq1}W_{k}=\bigcup_{k\geq2}F_{k}(X_{k})=F\left(\bigcup_{k\geq2}X_{k}\right)=F(X).
\]
Therefore, $F\colon(X,\mu)\to(W,\lambda)$ is an exact symplectomorphism.
\end{proof}

\section{Proof of Theorem~\ref{thm:general}}\label{sec:proof2}

The proof follows the same scheme as the proof of Theorem~\ref{thm:stab-convex}.
As in that proof let $W$ be a flexible Weinstein manifold which is formally symplectomorphic to $X$, $\{(X_i,\mu_i)\mid i\geq1\}$ and $\{(W_i,\lambda_i)\mid i\geq1\}$ exhaustions by Liouville and Weinstein subdomains of $X$ and $W$, respectively, and $(\varphi,\Phi_t)$ the formal symplectic embeddings of $(X,\mu)$ into $(W,\lambda)$.
 
We begin by proving Lemmas~\ref{lem:LW}, \ref{lem:WL}, and \ref{lem:ham} in the current context.

{\bf Step 1.}
{\em Construction of exact symplectic embeddings $f_{i}\colon(X_{i},\mu_{i})\to(W,\lambda)$, $i\geq 1$, in the formal isotopy class of $(\varphi|_{X_{i}},\Phi|_{TX_{i}})$.}

Denote $S_i:=\bigcap\limits_{t}Z_{\mu}^{-t}(X_i)$.
The attractor $S_i$ is compact and we have $S_i\subset\Core(X,\mu)$, and hence there exists an integer $N$ such that $S_i\subset \bigcup\limits_{j\leq N}C_j$.
 
Since the first subset $C_1$ is compact, there exists a finite open cover $\{U_{p_1},\ldots,U_{p_{k_1}}\}$ of $C_1$ such that each intersection $U_{p_j}\cap C_1$ admits a symplectic extension $\Sigma_{p_j}$ of positive codimension.
Choose also a cover $\{U'_{p_1},\ldots,U'_{p_{k_1}}\}$ such that $ U'_{p_j}\Subset U_{p_j}, j=1,\dots, k_1$.
Applying Theorem~\ref{thm:X-emb}~(\ref{thm:X-emb_non-para}) to $(\varphi|_{\Sigma_{p_1}},\Phi_t|_{\Sigma_{p_1}})$, we obtain an exact symplectic embedding $f^{(1)}_{1}\colon\Sigma_{p_1}\to W$ and moreover we can modify $\Phi_t$ so that $\Phi_1|_{T\Sigma_{p_1}}=df^{(1)}_{1}$.
The symplectic neighborhood theorem then allows us to extend $f^{(1)}_{1}$ to a neighborhood $\Omega_1\supset\Sigma_{p_1}$ in $U_{p_1}$.
We will keep the notation $f^{(1)}_{1}$ for this extension.
Denote $A_1:= \ol{C_1\cap U'_{p_1}}\cap U_{p_2}$.
Note that every point of $A_1$ has an access to infinity in $\Sigma_{p_2}$.
Hence, applying Theorem~\ref{thm:X-emb}~(\ref{thm:X-emb_rel}) we find a symplectic embedding $f^{(1)}_{2}\colon \Sigma_{p_2}\to W$ which coincides with $f^{(1)}_{1}$ on $\Op A_1\subset\Sigma_{p_2}$.
We further modify $\Phi_t$ so that $ \Phi_1|_{T\Sigma_{p_2}}=df^{(1)}_{2}$ and then use the symplectic neighborhood theorem to extend the exact symplectic embedding $f^{(1)}_{2}$ to a neighborhood $\Omega_2\supset\Sigma_{p_2}$ in $U_{p_2}$ so that the extended embedding $f^{(1)}_{2}$ coincides with $f^{(1)}_{1}$ on $\Op A_1$ in $U_{p_2}$.
Continuing this process we construct an exact symplectic embedding $f^{(1)}\colon\Op C_1\to W$.
Choosing a sufficiently small neighborhood $U_1\supset C_1$, where $f^{(1)}$ is defined, we find a finite open cover $\{U^{2}_{p_1},\ldots,U^{2}_{p_{k_2}}\}$ of the compact set $C_2\setminus U_1 $ such that each intersection $U^{2}_{p_j}\cap C_2$ admits a symplectic extension $\Sigma^{2}_{p_j}$ of positive codimension.
Repeating the above process inductively over elements of the cover $U^{2}_{p_j}$ of $C_2\setminus U_1$ and then continuing a similar process for $C_3,\dots,C_N$ we construct an exact symplectic embedding $h_i\colon\Op(C_1\cup\cdots\cup C_N)\to W$.
By our construction $Z_{\mu}^{-T}(X_i)$ for a sufficiently large $T$ is contained in a neighborhood of $\bigcup\limits_{j\leq N}C_j$ where $h_i$ is defined.
Hence the formula $f_i:=Z_{\lambda}^{T}\circ h_{i}\circ Z_{\mu}^{-T}$ defines the required exact symplectic embedding $f_i\colon X_i\to W$.
 
{\bf Step 2.}
{\em Construction of exact symplectic embeddings $g_i\colon(W_i,\lambda_i)\to(X,\mu)$, $i\geq1$, in the formal isotopy class of $(\varphi^{-1}|_{W_{i}},\Phi^{-1}|_{TW_{i}})$.}

This is a corollary of Theorem~\ref{thm:W-emb}~(\ref{thm:W-emb_non-para}), as in the case of Theorem~\ref{thm:stab-convex}.

{\bf Step 3.}
{\em Proof that the compositions $g_{k}\circ f_{k}\colon X_{k}\to X_{k+1}$ and $f_{k+1}\circ g_{k}\colon W_{k}\to W_{k+1}$ are Hamiltonian isotopic to the inclusions $\iota_{X_{k}}\colon X_{k}\to X_{k+1}$ and $\iota_{W_{k}}\colon W_{k}\to W_{k+1}$, respectively} (after readjusting the indices as in the proof of Theorem~\ref{thm:stab-convex}).

Steps 1 and 2 imply that $g_{k}\circ f_{k}$ is formally isotopic to $\iota_{X_{k}}=\varphi^{-1}|_{W_{k}}\circ\varphi|_{X_{k}}$.
To construct a genuine Hamiltonian isotopy connecting $g_{k}\circ f_{k}$ and $\iota_{X_{k}}$ we repeat the proof in Step 1, but using instead Theorems~\ref{thm:X-emb}~(\ref{thm:X-emb_1-para}) and \ref{thm:X-emb}~(\ref{thm:X-emb_rel_para}).
The existence of a Hamiltonian isotopy connecting $f_{k+1}\circ g_{k}$ and $\iota_{W_{k}}$ is even more straightforward using Theorem~\ref{thm:W-emb}~(\ref{thm:W-emb_1-para}).

{\bf Step 4}.
With the analogs of Lemmas~\ref{lem:LW}, \ref{lem:WL}, and \ref{lem:ham} established, we construct the required exact symplectomorphism $f\colon X\to W$ using the telescope construction exactly as in the proof of Theorem~\ref{thm:stab-convex}.~\hfill$\Box$

\end{document}